\documentclass{amsart}

\usepackage{amsmath}
\usepackage{latexsym}
\usepackage{amssymb}
\usepackage{setspace}
\usepackage{xcolor}
\usepackage{ifthen}
\usepackage{dsfont}
\usepackage{mathtools}
\usepackage{tikz}
\usepackage{hyperref}
\usetikzlibrary{positioning}

\newtheorem{Theorem}{Theorem}
\newtheorem{Thm}{Theorem}[section]
\newtheorem{Lem}[Thm]{Lemma}
\newtheorem{Con}[Thm]{Conjecture}

\newtheorem{Cor}[Thm]{Corollary}

\theoremstyle{definition}
\newtheorem{Def}[Thm]{Definition}
\newtheorem{Not}[Thm]{Notation}
\newtheorem{example}[Thm]{Example}

\theoremstyle{remark}
\newtheorem{Rem}[Thm]{Remark}

\numberwithin{equation}{section}

\newcommand{\ina}{infinite non-affine } 
\newcommand{\id}{\mathds{1}}

\newcommand{\EE}{\mathbb{E}}

\newcommand{\NN}{\mathbb{N}}

\newcommand{\bs}{\mathbf{s}}
\newcommand{\sib}{single braided }

\makeatletter
\@namedef{subjclassname@2020}{\textup{2020} Mathematics Subject Classification}
\makeatother

\usepackage{mathtools}

\DeclarePairedDelimiter\floor{\lfloor}{\rfloor}

\DeclareMathOperator{\Sym}{Sym}

\newboolean{ComsOn}
\setboolean{ComsOn}{true}
\ifthenelse{\boolean{ComsOn}}{\newcommand{\com}[2][blue]{\textcolor{#1}{#2}}}{\newcommand{\com}[2][blue]{}}

\definecolor{amethyst}{rgb}{0.6, 0.4, 0.8}
\definecolor{kellygreen}{rgb}{0.3, 0.73, 0.09}

\begin{document}

\title{Powers of Coxeter elements with unbounded reflection length}

\author{Marco Lotz}
\address{Fakultät für Mathematik, Otto-von-Guericke-
University Magdeburg, 
Universitätsplatz 2,
39106 Magdeburg}

\email{marco.lotz@ovgu.de}

\curraddr{Heidelberg University,
Faculty of Mathematics and Computer Science, 
Im Neuenheimer Feld 205,
69120 Heidelberg}

\thanks{The research of the author is supported in part by DFG Grant – 314838170, GRK 2297 MathCoRe.}


\date{\today.}



\begin{abstract}
For Coxeter groups with sufficiently large braid relations, we prove that the sequence of powers of a Coxeter element has unbounded reflection length. We establish a connection between the reflection length functions on arbitrary Coxeter groups and the reflection length functions on universal Coxeter groups of the same rank through the solution to the word problem for Coxeter groups. For Coxeter groups corresponding to a Coxeter matrix with the same entry everywhere except the diagonal, upper bounds for the reflection length of the powers of Coxeter elements are established.
\end{abstract}

\maketitle


\section*{Introduction}
Let $(W, S)$ be a Coxeter system. The conjugates of the standard generators in $S$ are called \emph{reflections}. We abbreviate the set of reflections with $R$. The length function corresponding to the set of reflections as a generating set is called \emph{reflection length}. For the reflection length function on direct products of finite and Euclidean reflection groups, formulas are known (see \cite{Carter1972} and \cite{Lewis2018}). In particular, the reflection length is a bounded function on these Coxeter groups. On the other hand, Duszenko shows in \cite{Duszenko2011} that the reflection length is an unbounded function on Coxeter groups that are not splitting into a direct product of finite and Euclidean reflection groups. This type of Coxeter groups is called \emph{\ina} Coxeter groups.  
Duszenko's proof is not constructive, which immediately raises follow-up questions: Which are the elements with high reflection length in \ina Coxeter groups? Is there a formula for the reflection length in \ina Coxeter groups?\par

There are no sequences of elements known with unbounded reflection length in \ina Coxeter groups in general. The only result obtained in this direction is due to Drake and Peters. They prove in \cite{Drake2021} that the reflection length of powers of the Coxeter elements in a universal Coxeter group of rank at least $3$ grows to infinity. Further, a formula for the reflection length for these elements exists (see Lemma~6 in \cite{Drake2021}).\par 
We prove that the sequence of powers of a Coxeter element has unbounded reflection length for every Coxeter group with sufficiently large off-diagonal entries in the Coxeter matrix. The minimal value aside the diagonal needs to be at least $5$ for rank three Coxeter groups and $3$ for rank four or higher Coxeter groups. Direct products of finite and Euclidean reflection groups are excluded indirectly in this way.

\begin{Theorem}\label{Thm: unbounded reflection length for powers of Coxeter elements}
Let $(W,S)$ be a Coxeter system of rank $n$ and let $M=(m_{ij})_{i,j\in I}$ denote its Coxeter matrix. Further, let $w$ be a Coxeter element in $W$. Then, 
\begin{itemize}
\item[(i)] if $n =3$ and $\min\{m_{ij}\mid i\neq j,\; i,j\in I\}\geq 5$, or
\item[(ii)] if $n\geq 4$ and $\min\{m_{ij}\mid i\neq j,\; i,j\in I\}\geq 3$,
\end{itemize}
we have 
\[
\lim_{\lambda\to\infty} l_R(w^\lambda) = \infty.
\]
\end{Theorem}

This work follows a new approach for the investigation of the reflection length. We compare the reflection length function of an arbitrary Coxeter group with the reflection length function of the universal Coxeter group of the same rank. For this, fix a word $\bs$ over the alphabet $S$ and compare the reflection length $l_R( \omega(\bs))$ of the element $\omega(\bs)$ represented by $\bs$ in the arbitrary Coxeter group and the reflection length $l_{R_n}( \omega_n(\bs))$ of the element $\omega_n(\bs)$ represented by $\bs$ in the universal Coxeter group. With this approach, we obtain the following relation between the different reflection length functions:

\begin{Theorem}\label{Thm: Lower bound}
Let $w$ be an element in a Coxeter system $(W,S)$ of rank $n$ represented by a $S$-reduced word $\bs = u_1\cdots \cdots u_p$. Further, let $\tilde{\bs}$ be a word obtained from $\bs$ by omitting all letters in a deletion set $D(\bs)$.  Let $m$ be the minimal number of braid-moves necessary to transform $\tilde{\bs}$ into the identity. 
The reflection length $l_{R}(w)$ in $W$ is bounded from below:
\[
l_{R_n}(\omega_n(\bs))-2 m \leq l_{R}(w).
\]
\end{Theorem}

A \emph{deletion set} for a word $\bs$ is a minimal set of letters such that an omission of these letters in $\bs$ results in the identity in $W$. A \emph{braid-move} is a substitution of a consecutive subword according to the relations of the form $(s_is_j)^{m({s_i,s_j})}= \id$. The proof of this result uses the solution to the word problem in Coxeter groups by Tits (see \cite{Tits1969}).
Theorem~\ref{Thm: unbounded reflection length for powers of Coxeter elements} is proven by counting all possible braid-moves on powers of Coxeter elements in order to use Theorem~\ref{Thm: Lower bound} and the formula for the reflection length in the universal case.\par 
For approaching the question of a formula for the reflection length in \ina Coxeter groups, we restrict ourselves to the study of Coxeter groups corresponding to a Coxeter matrix with the same entry everywhere except the diagonal. These Coxeter groups are called \emph{\sib}.
From counting subwords of the form $(s_is_j)^{m_{ij}}$ and with Theorem~\ref{Thm: Lower bound}, we derive sharp upper bounds for the reflection length of the powers of Coxeter elements in these Coxeter groups. The proof of the upper bounds is inductive. Finally, we conjecture that for a fixed word $\bs$ and an arbitrary Coxeter group $W$, there always exists a deletion set $D(\bs)$ such that $D(\bs)$ is a subset of a deletion set of $\bs$ in the universal Coxeter group of the same rank. This would imply a complete understanding of the relationship between the reflection length functions of an arbitrary Coxeter group and the universal Coxeter group of the same rank. We prove our conjecture for the special case, where $\bs$ represents a reflection in $W$.   \par \hfill

\paragraph{\bf{Structure}} This article has four sections. The preliminaries contain some foundations about the word problem and reflection length in Coxeter groups. These topics are combined in Section~\ref{Sec: Comparing}, where we establish the two main theorems and some technical lemmata refining the solution to the word problem.  For \sib Coxeter groups, we establish sharp upper bounds for the reflection length of powers of Coxeter elements in the third section. The last section deals with a conjecture about the general relation between the reflection length in an arbitrary Coxeter and the universal Coxeter group of the same rank.

\section{Preliminaries}
We recall the necessary background on Coxeter groups briefly. 
A Coxeter group is given by a finite complete graph without loops and an edge-labelling that is either a natural number greater than $1$ or infinity.
\begin{Def}
Let $\Gamma= (S, E)$ be a finite loop-free graph with finite vertex set $S= \{s_1, \dots, s_n\}$, edge set $E= \{\{u,v\}\subseteq S \mid u\neq v\}$ and an edge-labelling function $m:E\to \mathbb{N}_{\geq 2}\cup \{\infty\}$. The corresponding \textit{Coxeter group} $W$ is given by the presentation
\[
W = \langle S\mid s_i^2= (s_is_j)^{m({s_i,s_j})}= \id \;\forall\, i\neq j\in \{1,\dots, n\}\rangle.
\]
The pair $(W,S)$  is called \textit{Coxeter system}. The graph $\Gamma$ is called \textit{Coxeter graph}. Every Coxeter graph induces a presentation as above and vice versa. We abbreviate $m({s_i,s_j})$ with $m_{ij}$. The cardinality $|S|$ is called the \textit{rank} of $W$. The convention when drawing these graphs is to omit edges with label $2$ and to leave out $3$ as a label. 
If the labelling function is constantly infinity, we call the Coxeter group \textit{universal Coxeter group} and denote it with $W_n$.
Elements that can be written as $s_{\sigma(1)}\cdots s_{\sigma(n)}$ with $\sigma\in \Sym(n)$ are called \textit{Coxeter elements}.
\end{Def}

Prominent examples of Coxeter groups are spherical and Euclidean \emph{reflection groups}. These are discrete groups generated by finitely many hyperplane reflections in the $n$-dimensional sphere $\mathbb{S}^n$ or the $n$-dimensional Euclidean space $\EE^n$. The connected components of their Coxeter graphs are completely classified (see \cite{Coxeter1934a}).

\begin{Def}
We call a Coxeter group $W$ \textit{\sib} and denote it with $W^n_k$ if it is defined by a complete graph over $n$ vertices with a constant labelling function $m(\{u,v\})=k\in \NN_{\geq 2}$. 
\end{Def}

\begin{example}
The single braided Coxeter group $W^3_5$ corresponds to the Coxeter graph:
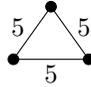
\[
\begin{tikzpicture}[scale=1, transform shape]
			\draw[fill=black]  (0,0) circle (2pt);
            \draw[fill=black]  (1,0) circle (2pt);
            \draw[fill=black]  (0.5,0.7) circle (2pt);
            \draw (0,0)--(1,0);
            \draw (1,0)--(0.5, 0.7);
            \draw (0,0)--(0.5, 0.7);
            \node at (0.5, -0.20) {$5$};
            \node at (0.06, 0.43) {$5$};
            \node at (0.94, 0.43) {$5$};
\end{tikzpicture}
\]
This is neither a finite nor an Euclidean reflection group according to the classification of these groups.
\end{example}

The basic theory of Coxeter groups is treated in detail in \cite{Humphreys1990}, \cite{Bjorner2005} and \cite{Davis2012}.\par
\hfill\par 

For every generating set of a group, there exists a corresponding length function.

\begin{Def}\label{Def: length function}
For a group $G$ with generating set $Y$, define $\bar{Y}:= Y\cup Y^{-1}$. The according length function $l_Y$ is defined as 
\[
l_Y: G\to\mathbb{N}\, ; \qquad g\mapsto \min \{n\in \NN\mid g\in \bar{Y}^n\}
\]
with $\bar{Y}^n= \{ y_1\cdots y_n\in G\mid y_i\in \bar{Y}\}$. The identity element $\id$ has length zero.
\end{Def}

In a Coxeter system $(W,S)$, every standard generator $s\in S$ is an involution.  We call minimal standard generator factorisations of an element $w\in W$ \textit{$S$-reduced}. Whether a factorisation is reduced, depends on the relations in the group.

\begin{Thm}[cf. Theorem~1 in \cite{Speyer2008}]\label{Thm: Speyer's Theorem}
Let $W$ be an infinite, irreducible Coxeter group and let $(s_1, \dots, s_n)$ be any ordering of generators in $S$. Then the word $(s_1\cdots s_n)^\lambda$
is $S$-reduced for any $\lambda\in \mathbb{N}$.
\end{Thm}

\subsection{The word problem for Coxeter groups}
From now on, let $(W,S)$ be a Coxeter system of rank $n$.
The following notation is used, where a distinction between a word in the free monoid $(S^*, \cdot)$ and the group element in $W$ it represents is important.
\begin{Not}
We abbreviate a word $\mathbf{s} = u_1\cdots u_p$ in $S^*$ with bold variables. Let $\omega: S^*\to W$ be the canonic surjection. We write $\omega(\mathbf{s})$ for the element in $W$ that is represented by the word $\mathbf{s}$.
\end{Not}

The word problem is the question, of whether there exists an algorithm that decides for every element $\bs\in S^*$ if $\omega(\bs)= \id$ in $W$ (cf. \cite{Dehn1911}). It is well known that the word problem in Coxeter groups is solvable. We use a specific theorem by Tits in \cite{Tits1969} leading to the solution of the word problem. Before stating Tits result, some definitions are necessary. 

\begin{Def}
A \textit{subword} of a word $\bs = u_1\cdots u_p$ over the alphabet $S$ is a subsequence $u_{i_1},\dots, u_{i_q}$ with $1\leq i_1<\cdots <i_q \leq p$ for all $j\in\{1,\dots q\}$. A subword is \emph{consecutive} if the subsequence ${i_1},\dots, {i_q}$ is consecutive.
\end{Def}

\begin{Def}\label{Def: nil- and braid-moves}
We distinguish two types of relations in the definition of a Coxeter group. 
Substituting the consecutive subword $s_is_i$ with the empty word $e$ in a word is called a \textit{nil-move}. For $m_{ij}\in \NN$, let $\mathbf{b}_{ij}$ be the word of length $m_{ij}$ consisting only of alternating letters $s_i$ and $s_j$ starting with $s_i$.
Substituting a consecutive subword $\mathbf{b}_{ij}$ with the subword $\mathbf{b}_{ji}$ in a word in $S^*$ is called a \textit{braid-move}.
\end{Def}

\begin{Rem}\label{Rem: Braid-moves do(n't) change number of letters in word of specific type}
Let $\mathbf{a}$ and $\mathbf{b}$ be two elements in $S^*$ and $s_i, s_j\in S$. From the relations of the type $s_i^2= \id$, the following nil-move can be derived 
\[
\omega(\mathbf{a}\cdot s_is_i\cdot \mathbf{b}) = \omega(\mathbf{a}\cdot\mathbf{b}).
\]
Equivalent to a braid relation $(s_is_j)^{m_{ij}}=\id$  is the braid-move
\[
\omega(\mathbf{a}\cdot \mathbf{b}_{ij}\cdot \mathbf{b}) = \omega(\mathbf{a}\cdot \mathbf{b}_{ji}\cdot\mathbf{b}).
\]
Braid-moves don't change the length of a word. Braid-moves on subwords $\mathbf{b}_{ij}$ of even $S$-length do not change the number of letters of a certain type in a word.  Braid-moves on subwords of odd $S$-length change the number of letters of a certain type by $\pm 1$.
\end{Rem}

\begin{Thm}[cf. Theorem 3 in \cite{Tits1969}]\label{Thm: solution word problem} 
Let $(W,S)$ be a Coxeter system and $w$ be an element in $W$. 
\begin{itemize}
\item[(i)] For every word $\bs\in S^*$ with $\omega(\bs)= w$, there exists a finite sequence of nil-moves and braid-moves that transforms $\bs$ into a $S$-reduced expression for $w$.
\item[(ii)] For every pair of $S$-reduced expressions for $w$, there exists a finite sequence of braid-moves that transforms one of the $S$-reduced expressions into the other.
\end{itemize}
\end{Thm}

\begin{Def}\label{Def: (MB)}
Let $(\alpha_i)_{i\leq a}$ be a finite sequence of nil-moves and braid-moves to transform a word $\bs$ into a $S$-reduced word in a Coxeter system $(W,S)$. We say that the sequence $(\alpha_i)_{i\leq a}$ is \emph{braid-minimalistic} if 
\begin{enumerate}
\item $\alpha_i$ is a braid-move if and only if $\alpha_{i-1}\circ\cdots \circ\alpha_1(\bs)$ has no consecutive subword of the form $ss$ for $s\in S$.
\item If $\alpha_i$ is a braid-move on a consecutive subword $\mathbf{b}_{ij}$, there is no other braid-move $\alpha_j$ in $(\alpha_i)_{i\leq a}$ on the subword in the same position as $\mathbf{b}_{ij}$ in $\bs$ involving the letters $s_i, s_j$.
\end{enumerate} This means especially that braid-moves just appear if there are no further nil-moves possible.
\end{Def}

\begin{example}
Consider the word $\bs:= s_1s_2s_1s_2s_2s_3s_2s_3$ and the \sib Coxeter group $W^3_3$ of type $\tilde{A}_2$. The first of the following two different ways of reducing $\bs$ in $W^3_3$ is braid-minimalistic. 
\begin{equation}
s_1s_2s_1s_2s_2s_3s_2s_3\mapsto s_1s_2s_1s_3s_2s_3\mapsto s_2s_1s_2s_3s_2s_3\mapsto s_2s_1s_3s_2
\end{equation}
\begin{equation}
s_1s_2s_1s_2s_2s_3s_2s_3\mapsto s_1s_1s_2s_1s_3s_2s_3s_3\mapsto s_2s_1s_3s_2
\end{equation} 
The second one is not braid-minimalistic since the nil-move $s_2s_2\mapsto e$ is not the first move.
\end{example}

\begin{Lem}\label{Lem: There always exists a sequence of nil and braid moves that is MB}
Let $\bs\in S^*$ be a word representing an element $\omega(\bs)$ in a Coxeter system $(W, S)$. There exists a finite, braid minimalistic sequence $(\alpha)_{i\leq a}$ of nil-moves and braid-moves transforming $\bs$ to a $S$-reduced expression for $\omega(\bs)$. 
\end{Lem}

\begin{proof}
From Theorem~\ref{Thm: solution word problem}, we know that a sequence of nil-moves and braid-moves transforms $\bs$ to a reduced expression.
The braid-minimalistic sequence $(\alpha)_{i\leq a}$ is obtained as follows: Execute nil-moves on $\bs$ and the resulting words until no more nil-moves are possible. Either the obtained word is $S$-reduced or there is a braid-move possible to obtain a word $\bs'$. Again, we know from the solution to the word problem that $\bs'$ is transformable to a $S$-reduced expression by nil-moves and braid-moves. So we repeat this procedure while keeping track of the executed braid-moves until we obtain a reduced word. This ensures that property (1) from Definition~\ref{Def: (MB)} holds for the obtained sequence. Property (2) in the definition holds since we keep track of executed braid-moves and don't execute redundant braid-moves.
\end{proof}

\subsection{Reflection length}
Now, we consider a Coxeter group $W$ with the conjugates of all standard generators as a generating set.
This is often called the \emph{dual approach}. This generating set is natural in the sense that the generators are exactly the elements that act as reflections on geometric realisations like the Davis complex or the Coxeter complex of the Coxeter group.

\begin{Def}
The conjugates of the standard generators in $S$ are called \textit{reflections}. A reflection is an involution and the set of reflections \[R:=\{wsw^{-1}\;|\; w\in W, \; s\in S\}\] is a generating set for the Coxeter group $W$.\par
According to Definition \ref{Def: length function}, there exists a corresponding length function $l_R$. We call minimal reflection factorisations of an element $w\in W$ \textit{$R$-reduced}.  Whether a factorisation is $R$-reduced, depends on the relations in the group.  The reflection length function of the universal Coxeter group $W_n$ is abbreviated with $l_{R_n}$.
\end{Def}

We recall some of the basic properties of reflection length.

\begin{Lem}\label{Lem: basic properties} 
Let $(W,S)$ be a Coxeter system  and $R$ be its set of reflections. We have
\begin{enumerate}
\item[(i)] $l_R(w)\leq l_S(w)$ for all $w\in W$.
\item[(ii)] The reflection length function $l_R$ is constant on conjugacy classes.
\item[(iii)] $l_R(g)-l_R(h)\leq l_R(gh)\leq l_R(g)+l_R(h)$ for all $g,h\in W$.
\item[(iv)]  $l_R(wr) = l_R(w) \pm 1$ for all $r\in R$ and all $w\in W$. 
\item[(v)] $l_R(w) \equiv l_S(w) \mod 2$ for all $w\in W$. 
\item[(vi)] If $W$ is a \sib Coxeter group of rank $n$, every element $\sigma$ in the symmetric group $S(n)$ defines a reflection length preserving automorphism of $W$ by permuting generators. 
\end{enumerate}
\end{Lem}

\begin{Rem}
The reflection length is additive on direct products. For a Coxeter group $W= W_1\times W_2$ that is the direct product of Coxeter groups $W_1$ and $W_2$, an element $(w_1,w_2)\in W$ has reflection length 
\[
l_R(w_1w_2) = l_{R_1}(w_1)+l_{R_2}(w_2)
\]
where $l_{R_1}$ and $l_{R_2}$ are the reflection length functions on $W_1$ and $W_2$.
\end{Rem}

\begin{example}
Consider the Coxeter groups $W_1$ and $W_2$ over three generators $a,b,c$ defined by the following two graphs from left to right:
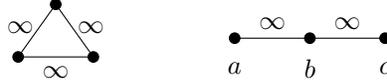
\[
\begin{tikzpicture}[scale=1, transform shape]
			\draw[fill=black]  (0,0) circle (2pt);
            \draw[fill=black]  (1,0) circle (2pt);
            \draw[fill=black]  (0.5,0.7) circle (2pt);
            \draw (0,0)--(1,0);
            \draw (1,0)--(0.5, 0.7);
            \draw (0,0)--(0.5, 0.7);
            \node at (0.5, -0.20) {$\infty$};
            \node at (0.03, 0.40) {$\infty$};
            \node at (0.97, 0.40) {$\infty$};
        \end{tikzpicture}
\qquad\qquad
\begin{tikzpicture}[scale=1, transform shape]
            \draw[fill=black] (0, 0) circle (2pt);
            \draw[fill=black] (1, 0) circle (2pt);
            \draw[fill=black] (2, 0) circle (2pt);
            \draw (0,0)--(2, 0);
            \node at (1.5, 0.2) {$\infty$};
            \node at (0, -0.4) {$a$};
            \node at (1, -0.4) {$b$};
            \node at (2, -0.4) {$c$};
            \node at (0.5, 0.2) {$\infty$};
\end{tikzpicture}
\]
The element represented by the word $\mathbf{w} = abcabc$ in $W_1$ has reflection length $4$ ( to be seen with Lemma~\ref{Lem: l_R in univ. group, powers of Cox element}). A $R$-reduced factorisation is $abcabc = aba\cdot aca\cdot b\cdot c\cdot$. In contrast, the element represented by $\mathbf{w}$ in $W_2$ has reflection length $2$. This is the minimal possible reflection length because of parity reasons (cf. Lemma~\ref{Lem: basic properties}). A $R$-reduced factorisation is $abcabc = aba\cdot cbc$.
\end{example}

\begin{Def}\label{Def: ina}
A Coxeter system $(W, S)$ is called \emph{\ina} if not all connected components of the corresponding Coxeter graph are of Euclidean or spherical type. This is equivalent to $W$ not splitting into a direct product of spherical and Euclidean reflection groups. 
\end{Def}

Since the reflection length is additive on direct products and formulas are known in Euclidean and spherical reflection groups, the reflection length function of all Coxeter groups except the \ina ones is well understood. On the other hand, the reflection length is an unbounded function on \ina Coxeter groups. 

\begin{Thm}[Theorem 1.1. in \cite{Duszenko2011}]\label{Thm: Duszenko LR unbpunded}
Let $(W,S)$ be an \ina Coxeter system. The reflection length $l_R$ is an unbounded function on $W$.
\end{Thm}

Apart from this, little is known about the reflection length function of \ina Coxeter groups.
The following result by Dyer is essential for computing reflection length in arbitrary Coxeter groups. It is the only known method to compute reflection length in \ina Coxeter groups in general. 

\pagebreak

\begin{Thm}[cf. Theorem 1.1. in \cite{Dyer2001}]\label{Thm: Dyer's Theorem}
Let $\mathbf{s} = u_1\cdots u_p$ be a $S$-reduced expression in a Coxeter group $W$. Then $l_R(\omega(\mathbf{s}))$ is
the minimum of the natural numbers $q$ for which there exist $1\leq i_1 < \cdots < i_q \leq p$ such that $\id = \omega(u_1 \cdots \hat{u}_{i_1}\cdots \hat{u}_{i_q}\cdots u_p)$, where a hat over a letter indicates its omission.
\end{Thm}

\begin{example}\label{Example: Counter example to lower bound theorem}
Consider the Coxeter group $W$ of type $\tilde{A}_2$ corresponding to the graph 
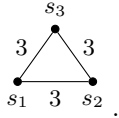
\[
\begin{tikzpicture}
\small
 \tikzset{enclosed/.style={draw, circle, inner sep=0pt, minimum size=.1cm, fill=black}}

      \node[enclosed, label={below,yshift=-0.0cm: $s_1$}] (A) at (0,0) {};
      \node[enclosed, label={below: $s_2$}] (B) at (1,0) {};
      \node[enclosed, label={above, yshift=0cm: $s_3$}] (C) at (0.5,0.7) {};

      \draw (A) -- (B) node[midway, below] (edge1) {$3$};
      \draw (C) -- (B) node[midway, right, ,yshift=+0.1cm] (edge2) {$3$};
      \draw (C) -- (A) node[midway, left, ,yshift=+0.1cm] (edge3) {$3$};
\end{tikzpicture}.
\]
Let $\mathbf{w}:=s_1s_2s_3\in S^*$. The reflection length of $\omega({\mathbf{w}^4s_1s_2})$ is $2$. This is to be seen by omitting the following letters:
\begin{align*}
& s_1s_2s_3s_1\hat{s}_2s_3s_1s_2s_3\hat{s}_1s_2s_3s_1s_2 = s_1s_2\cdot s_3s_1s_3s_1\cdot s_2s_3s_2s_3\cdot s_1s_2\\ 
=\; & s_1s_2\cdot s_1s_3\cdot s_3s_2\cdot s_1s_2 = \id.
\end{align*}
Hence, by Theorem \ref{Thm: Dyer's Theorem} the reflection length is at most $2$. Since the word length is even, this implies $l_R( w^4s_1s_2)=2$.
\end{example}

\begin{Def}
In a Coxeter system $(W, S)$, let $\mathbf{w}= u_1\cdots u_p\in S^*$ be a word representing an element $w= \omega(\mathbf{w})\in W$. We call a minimal set of indices like in Theorem~\ref{Thm: Dyer's Theorem} \emph{deletion set} for $\mathbf{w}$ and abbreviate such a set with $D(\mathbf{w})$.
Thus, we have for the cardinality $|D(\mathbf{w})|= l_R(w)$.
\end{Def}

\begin{example}
Consider the single braided Coxeter group  $W^3_4$  (with $m_{ij}=4$ and rank $3$). Define $\mathbf{w}:=s_1s_2s_3 \in S^*$. The reflection length of $\omega(\mathbf{w}^5s_1s_2)$ is $5$ and there are multiple deletion sets. 
\begin{enumerate}
\item $s_1s_2\hat{s}_3s_1s_2\hat{s}_3s_1s_2s_3s_1\hat{s}_2s_3s_1\hat{s}_2s_3s_1s_2 = s_2s_1\cdot s_1s_3s_2$ has reflection length $1$.
\item $s_1\hat{s}_2{s}_3s_1s_2\hat{s}_3s_1s_2\hat{s}_3s_1s_2\hat{s}_3s_1{s}_2s_3s_1\hat{s}_2 = \id$.
\end{enumerate}
\end{example}

\begin{Lem}\label{Lem: After Dyer}
Let $w$ be an element of a Coxeter system $(W,S)$ represented by a word $\mathbf{s}=u_1\cdots u_p\in S^*$ as above and let $D(\bs)=\{i_1, \dots, i_q\}$ be a deletion set. For every proper subset $N=\subsetneq D(\bs)$, let $w^{\setminus N}$ be the element represented by the word that we obtain from $\bs$ by removing all letters with indices in $N$. With $w':=w^{\setminus N}$ we have
\[
l_R(w')= q -|N|\quad \text{ and }\quad l_R(w^{\setminus N\cup\{ i_j\}})= l_R(w')-1 \quad \forall\; i_j\in \{i_1, \dots, i_q \}\setminus N.
\]
\end{Lem}
\begin{proof}
Define the reflection $r_i:= u_1\cdots u_{i-1}u_iu_{i-1}\cdots u_1$. A deletion set $D(\bs)$ of indices $1\leq i_1 < \cdots < i_q \leq n$ corresponds to the following minimal reflection factorisation of $w$: 
\[
w= r_{i_q}\cdots r_{i_1}.
\]
A proper subset $N\subsetneq D(\bs)$ is a totally ordered set $i_{n_1}<\dots <i_{n_m}$ with $m< q$. To remove all the letters corresponding to the indices $i_{n_j}\in N$ from the word $\bs$ we multiply from the left with $r_{i_{n_1}}\cdots r_{i_{n_m}}$. This results in 
\[
\omega(u_1\cdots \hat{u}_{i_{n_j}}\cdots u_p)= r_{i_{n_1}}\cdots r_{i_{n_m}}\cdot w=r_{i_{n_1}}\cdots r_{i_{n_m}}\cdot r_{i_q}\cdots r_{i_1}.
\]
For every $r_{i_{n_j}}$ exits an $r_{i_k}$ such that $r_{i_{n_j}}=r_{i_k}$, because $N\subsetneq D(\bs)$. Thus, we have 
\[
w^{\setminus N}= r_{i_{n_1}}\cdots r_{i_{n_m}}\cdot r_{i_q}\cdots r_{i_1}= r_{i_{n_1}}\cdots r_{i_{n_m}}\cdot r_{i_q}\cdots r_{i_{n_m}}\cdots r_{i_{n_1}}  \cdots r_{i_1}. 
\]
 In the reflection factorisation $r_{i_{n_1}}\cdots r_{i_{n_n}}\cdot r_{i_p}\cdots r_{i_{n_n}}\cdots r_{i_{n_1}}  \cdots r_{i_1}$ all reflections between two equal reflections $r_{i_{n_j}}$ can be written as reflections conjugated with $r_{i_{n_j}}$. The reflection length only changes by $\pm 1$ when multiplying with a reflection (see Lemma~\ref{Lem: basic properties}) and we conclude $l_R(w^{\setminus N})= l_R(w)-|N|$. With the same arguments we obtain $l_R(w^{\setminus N\cup \{i_j\}})= l_R(w^{\setminus N})-1$ for any $i_j\in \{i_1, \dots, i_q \}\setminus N$.
\end{proof}

For the powers of Coxeter elements in the universal Coxeter group, there exists a formula for the reflection length. 

\begin{Lem}[Lemma 7 in \cite{Drake2021}]\label{Lem: l_R in univ. group, powers of Cox element}
In the universal Coxeter group $W_n$ with $w= s_1\cdots s_n$ we have
\[
l_{R_n}(w^\lambda\cdot s_1\cdots s_r)= \lambda(n-2)+r
\]
for $1\leq r\leq n$ and $\lambda\in \NN_0$. 
\end{Lem}

 Therefore, the powers of Coxeter elements are the first candidates to find unbounded reflection length in other Coxeter groups as well. This is our starting point for comparing the reflection length functions of different Coxeter groups with the one in the universal Coxeter group of the same rank.

\section{Comparing reflection length in arbitrary and universal Coxeter groups}\label{Sec: Comparing}
In this section, we compare the reflection length function of an arbitrary Coxeter group with the reflection length function of the universal Coxeter group of the same rank. The braid relations are crucial for this. In comparison to the universal Coxeter group, the reflection length of an element possibly decreases by applying braid relations in an arbitrary Coxeter group.

\begin{Not}
Given a generating set $S$ with $n$ elements, we write $\omega_n$ for the canonic surjection $S^*\to W_n$.
\end{Not}

\begin{Lem} \label{Lem: l_Rn upper bound for l_R}
Let $(W,S)$ be an arbitrary Coxeter system of rank $n$ and let $R$ be the set of reflections in $W$. For any element $v\in W$ represented by a $S$-reduced word $\bs\in S^*$ the reflection length is bounded by
\[
l_R(v)\leq l_{R_n}(\omega_n(\bs)). 
\]

\end{Lem} 
\begin{proof}
Let $\omega_n(\mathbf{r}_1\cdots \mathbf{r}_l)= \omega_n(\bs)$ be a $R_n$-reduced reflection factorisation in $W_n$. The absence of braid relations in universal Coxeter groups implies that the word $\mathbf{r}_1\cdots \mathbf{r}_l$ can be transformed to $\bs$ with a sequence of nil-moves and  that the words $\mathbf{r}_i\in S^*$ are palindromes (see Theorem \ref{Thm: solution word problem}). So we have $\omega(\mathbf{r}_i)\in R$ for all $\mathbf{r}_i$ and $v = \omega(\mathbf{r}_1)\cdots \omega(\mathbf{r}_l)$. 
\end{proof}

\begin{Lem}\label{Lem: word representing 1 with briad move contains (s_is_j)^{m_{ij}}}
Let $(W, S)$ be a Coxeter system with presentation $\langle S \mid \mathcal{R} \rangle$. Let $\bs\in S^*$ be a word with $\omega(\bs) = \id$. If a braid-minimalistic sequence of nil-moves and braid-moves $(\alpha_i)_{i\leq a}$ transforming $\bs$ into the empty word $e$ contains a braid-move, $\bs$ contains a subword of the form $(s_is_j)^{m_{ij}}$ with $(s_is_j)^{m_{ij}}$ in $\mathcal{R}$.
\end{Lem}
\begin{proof}
We prove the assertion by induction over the number of braid-moves in $(\alpha_i)_{i\leq a}$. If $(\alpha_i)_{i\leq a}$ contains exactly one braid-move $\mathbf{b}_{12}\mapsto\mathbf{b}_{21}$, every letter in $\mathbf{b}_{21}$ is cancelled by a nil-move with a letter in $\bs$ outside of $\mathbf{b}_{21}$. Hence, $\bs$ either contains $(s_1s_2)^{m_{12}}$ or $(s_2s_1)^{m_{21}}$ as a subword. Additionally, the braid-move $\mathbf{b}_{12}\mapsto\mathbf{b}_{21}$ is executed on a subword of $(s_1s_2)^{m_{12}}$ or $(s_2s_1)^{m_{21}}$.\par
Assume that $(\alpha_i)_{i\leq a}$ contains $(n+1)$ many braid-moves and is braid-minimalistic. Let $\alpha_b$ be the first braid-move in the sequence. The sequence $(\alpha_i)_{b<i\leq a}$ transforms the word $\bs' :=\alpha_b\circ\cdots\circ\alpha_1(\bs)$ into the empty word and contains $n$ braid-moves. According to the induction assumption, $\bs'$ contains a subword $\mathbf{r}$ of the form $(s_is_j)^{m_{ij}}$ with $(s_is_j)^{m_{ij}}$ in $\mathcal{R}$.  Braid-moves on subwords of even $S$-length do not change the number of letters of a certain type.  Braid-moves on subwords of odd $S$-length change the number of letters of a certain type by $\pm 1$ (see Remark~\ref{Rem: Braid-moves do(n't) change number of letters in word of specific type}). In general, a word $\mathbf{b}_{ij}$ contains at least one $s_i$ and one $s_j$. So in case $\alpha_b$ is not a braid-move solely on letters of $\mathbf{r}$, we can conclude directly that $\bs$ contains a subword of the form $(s_is_j)^{m_{ij}}$ with $(s_is_j)^{m_{ij}}$ in $\mathcal{R}$.\par
Consider the case where $\alpha_b$ is a braid-move solely on letters of $\mathbf{r}$. Following the induction hypothesis, there exists another braid-move $\alpha_j$ in $(\alpha_i)_{i\leq a}$ on a subword of $\mathbf{r}$. Together with being braid-minimalistic, this implies that one of the braid-moves $\alpha_b$ and $\alpha_j$ acts on the first half of $(s_is_j)^{m_{ij}}$ and one on the second. There exists a letter $s$ in $\mathbf{r}$ that is not touched by $\alpha_b$ and is adjacent to a letter touched by $\alpha_b$ in $\mathbf{r}$. In $\alpha_{b-1}\circ\cdots\circ\alpha_1(\bs)$ this letter $s$ can't be adjacent to a letter touched by $\alpha_b$, because this contradicts being braid-minimalistic. Let $\bar{\bs}$ be the consecutive subword of $\alpha_{b-1}\circ\cdots\circ\alpha_1(\bs)$ separating $s$ and a letter that is touched by $\alpha_b$.\par 
 If $\omega(\bar{\bs}) =\id$ via a subsequence of $(\alpha_i)_{i\leq a}$, we can apply the induction hypothesis. Correspondingly, $\bar{\bs}$ contains a subword of the form $(s_ks_l)^{m_{kl}}$ with $(s_ks_l)^{m_{kl}}$ in $\mathcal{R}$. Since $\bar{\bs}$ is a subword of $\bs$, the same holds for $\bs$.\par
If there is no subsequence of $(\alpha_i)_{i\leq a}$ that transforms $\bar{\bs}$ into $e$, there are no nil-moves  within the letters of $\mathbf{r}$. Since $\omega(\bs) = \id$, the letters in $\mathbf{r}$ have to be cancelled by nil-moves in pairs with letters outside of $\mathbf{r}$. The properties of possible braid-moves on these letters outside of $\mathbf{r}$ imply as before that $\bs$ contains a subword of the form $(s_is_j)^{m_{ij}}$ in this case, too.
\end{proof}

\begin{Lem}\label{Lem: r_R = univ. CoxGrp.}
Let $(W,S)$ be an arbitrary Coxeter system of rank $n$ with presentation $\langle S \mid \mathcal{R} \rangle$ and relations $\mathcal{R}$. Further, let $l_R$ be the reflection length function in $W$ and $v\in W$ be an element represented by a $S$-reduced word $\bs$.
If $v$ has no $S$-reduced expression that contains subwords of the form $(s_is_j)^{m_{ij}}$ 
for all $(s_is_j)^{m_{ij}}$ in $\mathcal{R}$, the reflection length of $v$ is 
\[
l_R(v)=l_{R_n}(\omega_n(\bs)). 
\]
 
\end{Lem}
\begin{proof}
Let $D(\bs)$ be a deletion set.  
Let $\bs'$ be the word we obtain from $\bs$ by deleting the letters in $D(\bs)$. It is $\omega(\bs') = \id$ in $W$. According to Theorem \ref{Thm: solution word problem}, there exists a sequence of nil-moves and braid-moves that transforms $\bs'$ into the empty word. \newline
Under the assumptions, we show that there exists such a sequence without a braid-move. This implies the lemma because nil-moves are also allowed in $W_n$ and for all $\mathbf{a}\in S^*$ we have $l_R(\omega(\mathbf{a}))\leq l_{R_n}(\omega_n(\mathbf{a}))$ (see Lemma \ref{Lem: l_Rn upper bound for l_R}).\par
Fix such a braid-minimalistic sequence $(\alpha_i)_{i\leq a}$ and
assume that it contains a braid-move (cf. Lemma~\ref{Lem: There always exists a sequence of nil and braid moves that is MB}).
 With Lemma~\ref{Lem: word representing 1 with briad move contains (s_is_j)^{m_{ij}}}, we conclude that $\bs'$ contains a subword of the form $(s_is_j)^{m_{ij}}$ with $(s_is_j)^{m_{ij}}$ in $\mathcal{R}$. Hence, the same holds for $\bs$ and we arrive at a contradiction. Correspondingly, the sequence $(\alpha_i)_{i\leq a}$ contains no braid-moves. This proves the assertion.

\end{proof}

\begin{Def}
The \textit{indicator function} $\id_A$ of a subset $A$ of a set $X$ is defined as
\[
\id_A:X\to \{0,1\};\quad x\mapsto\
		\begin{cases}
			1 & x\in A\\
			0 & x\notin A\\
		\end{cases}.
\]
For convenience, we write for example $\id_{\{2,3,\dots\}}(a)$ as $\id_{a\geq 2}$ for $a \in \mathbb{Z}$ and do so analogously in similar cases.
\end{Def}

Together with Lemma \ref{Lem: r_R = univ. CoxGrp.}, the Lemma above implies that the reflection length of powers of a Coxeter element in arbitrary Coxeter systems behaves like in the universal case, if the power of the Coxeter element is small enough in relation to the braid relations of the Coxeter group.

\begin{Cor}\label{Cor: l_R= l_Rn in our Situation}
Let $(W,S)$ be an arbitrary Coxeter system of rank $n$ with  reflection length function $l_R$ and $\mathbf{w}= s_1\cdots s_n$. The following equality holds
\[
l_R(\omega(\mathbf{w}^\lambda \cdot s_1\cdots s_r))=l_{R_n}(\omega_n(\mathbf{w}^\lambda \cdot s_1\cdots s_r)) = \lambda(n-2)+r
\]
if $\lambda+\id_{\{r\geq 2\}}<\min \{m({s_i, s_j})\mid s_i,s_j\in S\}$. 
\end{Cor}

Lemma \ref{Lem: r_R = univ. CoxGrp.} shows that if an element has a $S$-reduced expression without subwords of the form $(s_is_j)^{m_{ij}}\in \mathcal{R}$, the reflection length of the element represented by a $S$-reduced word in an arbitrary Coxeter group is equal to the reflection length of the element represented by the same word in the universal Coxeter group of the same rank. Now, we consider arbitrary words. So the corresponding elements in the Coxeter groups may have different reflection lengths.

We are now ready to prove Theorem~\ref{Thm: Lower bound}:

\begin{proof}[Proof of Theorem~\ref{Thm: Lower bound}] 
We have $w(\tilde{\bs}) = \id$. Thus, a finite sequence $(\sigma_1, \dots, \sigma_t)$ of nil-moves and braid-moves transforms the word $\tilde{\bs}$ into the empty word $e$ (see Theorem~\ref{Thm: solution word problem}). There exist no braid-moves in $W_n$. So if there exists a sequence of only nil-moves that transforms $\tilde{\bs}$ into $e$, we have $w_n(\tilde{s})=\id$ and therefore $l_R(x)= l_{R_n}(w(\bs))$.\par
Assume that  $(\sigma_1, \dots. \sigma_t)$ contains exactly $m$ braid-moves. This translates into a finite sequence of words 
\[
(\bs_1 = \bs'_1 \mathbf{b}_{i_1j_1} \bs''_1, \dots, \bs_m = \bs'_m \mathbf{b}_{i_mj_m} \bs''_m, \bs_{m+1}),
\]
where $\bs_1$ is obtained from $\tilde{\bs}$ by finitely many nil-moves, $\bs_{l+1}$ is obtained from $\bs_{l}$ by applying the braid-braid move $\mathbf{b}_{i_lj_l}\mapsto \mathbf{b}_{j_l i_l}$ and finitely many nil-moves. The last entry $\bs_{m+1}$ is transformable into the empty word only by nil-moves. This is why we have $\omega(\bs_{m+1})) = \omega_n(\bs_{m+1})= \id$. \par
On the level of group elements, we have $\omega(\tilde{\bs})= \omega(\bs_i)= \id$ and in the universal Coxeter group $W_n$
\[
\omega_n(\bs_{l+1}) = \omega_n(\bs_l)\cdot r_lr'_l,
\]
where $r_l,r'_l$ are reflections in $W_n$. Executing a braid-move on a word is equivalent to removing the first (most left) letter and adding one letter on the right side of $\mathbf{b}_{ij}$: 
\[
s_is_j\cdots s_i \mapsto \hat{s}_is_j\cdots s_i\cdot s_j.
\]
Removing as well as adding one letter translates to multiplying with a reflection from the left or right on the level of group elements.
Applying Lemma~\ref{Lem: basic properties} for every transition between elements in the sequence $\omega_n(\bs_1), \dots, \omega_n(\bs_{m+1}))$, we obtain 
$
l_{R_n}(\omega_n(\tilde{\bs}))\leq 2m. 
$
Together with Theorem \ref{Thm: Dyer's Theorem}, this implies
\[
l_{R_n}(\omega_n(\bs))\leq l_R(\omega(\bs))+2m,
\]
which is equivalent to the assertion.
\end{proof}

\begin{Rem}
Even though this lower bound is sharp in some cases, it isn't in general. Take for example the \sib Coxeter group $W^3_3$ with generating set $S=\{s_1,s_2,s_3\}$. The element $v\in W^3_3$ represented by the word 
\[
\mathbf{r}  := s_1s_2s_1s_3s_1s_3s_2s_1s_2\in S^*
\]
has reflection length $1$, where $l_{R_n}(\omega_n(\mathbf{r})) = 3$. This can be seen by removing the letter in the middle (see Theorem~\ref{Thm: Dyer's Theorem}). For $W^3_3$, we obtain the identity. For $W_n$, we don't. Thus, the reflection length $l_{R_n}(\omega_n(\mathbf{r}))$ has to be at least 3, because of parity reasons. The sequence of word transformations is
\begin{align*}
s_1s_2s_1s_3\hat{s}_1s_3s_2s_1s_2\mapsto s_1s_2s_1\cdot s_2s_1s_2\mapsto s_2s_1s_2 \cdot s_2s_1s_2\mapsto e.
\end{align*}
Here, the last $\mapsto$ includes 3 nil-moves. According to the theorem, we have 
\[
l_{R_n}(\omega_n(\mathbf{r}))-2\leq l_R(v)
\]
and inserting the values leads to $1\leq 1$. So in this case, the bound is sharp.\par
On the other hand, in Example \ref{Example: Counter example to lower bound theorem} there are $3$ braid-moves executed for the element $w^4s_1s_2\in W^3_3$. We know from Lemma \ref{Lem: l_R in univ. group, powers of Cox element} that $l_{R_n}(\omega_n((s_1s_2s_3)^4s_1s_2))=6$. Consequently, the lower bound from the theorem is $0$ but the reflection length of $w^4s_1s_2$ is $2$.
\end{Rem}

\begin{Rem}
The lower bound also depends on the deletion set $D(\bs)$. Consider the word $\mathbf{t}= s_3s_1s_2s_1s_3s_2s_1s_2$ in a Coxeter group $W = \langle\{s_1,s_2,s_3\}\mid \mathcal{R}\rangle$, in which $(s_1s_2)^3 \in \mathcal{R}$ and no braid relation exists with $m_{ij}=2$. It is $l_R(\omega(\mathbf{t})) = l_{R_n}(\omega_n(\mathbf{t}))$ and the reflection length is $2$, because of parity reasons and the following deletion sets 
\[
\hat{s}_3s_1s_2s_1\hat{s}_3s_2s_1s_2\qquad\text{and}\qquad s_3s_1\hat{s}_2s_1s_3s_2\hat{s}_1s_2.
\]
Applying the theorem above with the left set leads to a lower bound of $0$ whereas the right side leads to a sharp lower bound of $2$. 
In the case of the word $\mathbf{r}$ from the last remark, the lower bound can't be sharpened by a different choice of a deletion set, because the only deletion set for a reduced word representing a reflection is the letter in the middle.\par
An immanent question is under which circumstances the lower bound is sharp. Follow-up questions are if the lower bound is sharp for special deletion sets  and if the sharpness only holds for certain elements. These questions won't be discussed further in this work.
\end{Rem}

The lower bound of Theorem \ref{Thm: Lower bound} can be improved by the following lemmas. The lemmas are based on the observation that a finite sequence of braid-moves is in some cases equivalent to the concatenation with 2 reflections.

\pagebreak

\begin{Lem}\label{Lem: braid moves interfering in one letter}
Let $(W,S)$ be a Coxeter system. Further, let $\sigma =\sigma_n \circ\dots\circ\sigma_1$ be a sequence of braid-moves on a word $\bs\in S^*$ with $\sigma_i : \mathbf{b}^i_{x_iy_i}\mapsto \mathbf{b}^i_{y_ix_i}$ for $x_i, y_i\in S$. If the last letter of $\mathbf{b}^i_{y_ix_i}$ is the first letter of $\mathbf{b}^{i+1}_{x_{i+1}y_{i+1}}$ for all $i\in \{1,\dots, n-1\}$, then the execution of $\sigma$ on the word $\bs$ is equivalent to multiplying with two reflections on a group-element-level in $W_n$:
\[
\omega(\sigma(\bs)) = \omega(\bs), \;\; \omega_n(\sigma(\bs)) = \omega_n(\bs)\cdot r_1\cdot r_2\quad\text{with}\quad r_i\in R_n.
\]
\end{Lem}

\begin{proof}
Executing the braid move $\sigma_i$ is the same as omitting the first letter (letter on the left side) of $\mathbf{b}^i_{x_iy_i}$ and extending $\mathbf{b}^i_{x_iy_i}$ with a letter $s\in \{x_i,y_i\}$ on the right side. If the corresponding $m_{ij}$ is even, we have $\mathbf{b}_{ij}\mapsto s_i\cdot \mathbf{b}_{ij}\cdot s_i = \mathbf{b}_{ji}$. If it is odd, we have $\mathbf{b}_{ij}\mapsto s_i\cdot \mathbf{b}_{ij}\cdot  s_j = \mathbf{b}_{ji}$. Since  the last letter of $\mathbf{b}^i_{y_ix_i}$ is the first letter of $\mathbf{b}^{i+1}_{x_{i+1}y_{i+1}}$ for all $i\in \{1,\dots, n-1\}$, the letter that gets inserted by applying $\sigma_i$ is removed by applying $\sigma_{i+1}$ for all $i\in \{1,\dots, n-1\}$. Hence, executing $\sigma$ in total corresponds to omitting the letter for $\sigma_1$ and inserting the letter for $\sigma_n$. On a group-element-level in $W_n$, this is equivalent to multiplying with two reflections that are represented by palindromes.
\end{proof}

\begin{Rem}\label{Rem: braid moves interfering in one letter}
The lemma is also true if the first letter of $\mathbf{b}^i_{y_ix_i}$ is the last letter of $\mathbf{b}^{i+1}_{x_{i+1}y_{i+1}}$ for all $i\in \{1,\dots, n-1\}$. This condition reflects the analogous situation from the right to the left. Whereas in the presuppositions of the lemma, a letter is ``wandering'' from left to right from braid-move to braid-move.\par
Another situation to be considered is that the first and the last letter of the first braid-move are involved in following braid-moves in both directions. In this case, the braid-moves in one direction don't influence the ones in the other direction and can be executed first. So the lemma can be applied in one direction. Afterwards, the braid-moves in the other direction are considered. Analogously to the lemma, we obtain that executing the braid-moves in both directions is equivalent to multiplying with four reflections on a group-element-level in $W_n$.
\end{Rem}

\begin{Lem}
Let $(W,S)$ be a Coxeter system. Further, define $\mathbf{B}^{-2}_{ij}$ to be the alternating word consisting of the two letters $s_i,s_j\in S$ starting with $s_i$ and with word length $2m_{ij}-2$. 
Let $s_1,s_2, s_3\in S$ be distinct.  The concatenation of two distinct braid-moves on the word $\mathbf{B}^{-2}_{12}\mathbf{B}^{-2}_{31}$ :
\[
\mathbf{B}^{-2}_{12}\mathbf{B}^{-2}_{31}\mapsto  s_2s_3
\]
is  equivalent to multiplying with two reflections on a group-element-level in $W_n$.
\end{Lem}

\begin{proof}
Since both alternating words $\mathbf{B}^{-2}_{12},\mathbf{B}^{-2}_{31}$ have even word length and consist exactly of two distinct letters, $\mathbf{B}^{-2}_{12}$ ends in $s_2$. Ignoring the first or the last letter   of them yields to palindromes of odd word length. Deleting these palindromes is equivalent on a group-element-level in $W_n$ to multiplying with a reflection for each palindrome. The word that is left is $s_2s_3$.
\end{proof}

\begin{Rem}
With this lemma and Theorem \ref{Thm: Lower bound} we can derive a sharp lower bound in the case of Example \ref{Example: Counter example to lower bound theorem}.
\end{Rem}

\begin{Not}\label{Not: bracket notation for elements}

Let $(W, S)$ be a Coxeter system. We write a word $\bs := u_1\cdots u_m\in S^*$ as $\bs = (\bs(1), \dots, \bs(m))$ with $u_i =\bs(i)$ to keep track of the initial position of a letter while applying nil-moves and braid-moves. In the process of applying nil-moves and braid-moves, the $i$-th position can be filled with the according letter $\bs(i)\in S$ or the empty word $e$. A subword $\bs(i_1)\cdots \bs(i_t)$ is called \emph{consecutive} if the entries in $(\bs(1), \dots, \bs(m))$ between the letters $\bs(i_a)$ and $\bs(i_{a+1})$ are all filled with $e$ for all $a\in \{1,\dots, n-1\}$. With this notation, nil-moves change two equal consecutive letters to the empty word. A braid-move on a consecutive subword $\tilde{\bs}$ that consists of two alternating letters $s_i, s_j\in S$ with length $m_{ij}$ permutes the two types of letters in $\tilde{\bs}$.
\end{Not}

\begin{Def}
Two braid-moves in a sequence of nil-moves and braid-moves on a word $\bs = (\bs(1),\dots,\bs(m))$ \emph{interfere} if there exists an $i \in \{1,\dots, m\}$ such that both braid-moves change the $i$-th entry in $(\bs(1),\dots,\bs(m))$. In this case, we say that the braid-moves are interfering in the index $i$.
\end{Def}

Using this vocabulary and notation, Theorem~\ref{Thm: Lower bound} can be strengthened with the following lemma. 

\begin{Lem}\label{Lem: Always exists sequence with braid moves not interferring more than in one index}
Let $(W,S)$ be a Coxeter system with $m_{ij}>2$ for all distinct $s_i,s_j\in S$. For every word $\bs = (\bs(1),\dots,\bs(m))$ over the alphabet $S^*$, there exists a finite sequence of nil-moves and braid-moves $(\alpha)_{l\in L}$ transforming the word into an $S$-reduced word such that all pairs of braid-moves in $(\alpha)_{l\in L}$ interfere maximally in one index.
\end{Lem} 
\begin{proof}
Let $\bs = (\bs(1),\dots,\bs(m))$ be a word over the alphabet $S^*$. From the solution of the word problem for Coxeter groups, we know that there exists a sequence $(a'_k)_{k\in K}$ that transforms $\bs = (\bs(1),\dots,\bs(m))$ into a $S$-reduced word with some $\mathbf{s}(i) = e$ if $\bs$ isn't $S$-reduced (see Theorem~\ref{Thm: solution word problem}). If $\bs$ is already $S$-reduced, the sequence is empty. Since there is no braid-move necessary to reduce a word of word length $2$, the statement of the lemma holds if only one or two indices are different from $e$. Define $c:=|\{i\in \{1,\dots,m\}\mid \bs(i)\neq e\}|$ to be the number of indices in $\bs$ that are not filled with the empty word. From here, we prove the lemma by induction over $c$.\par
Let $\bs = (\bs(1),\dots,\bs(m))$ be a word with $|\{i\in \{1,\dots,m\}\mid \bs(i)\neq e\}| = c+1$. If $\bs$ is $S$-reduced, there exists the empty sequence as a sequence that does not contain pairs of braid-moves that interfere in more than one index. Else, there exists a finite sequence $(a'_k)_{k\in K}$ of nil-moves and braid-moves that transforms $\bs$ into a $S$-reduced word. Assume that $(a'_k)_{k\in K}$ contains two braid-moves $a'_x$ and $a'_y$ with $x<y$ such that $a'_y$ interferes with $a'_x$ in at least two indices. Further, we may assume that no nil-moves are appearing in the sequence $(a'_k)_{k\in K}$ before $a_x$, because in this case a nil-move lowers $c+1$ by $-2$ and following the induction assumption $\bs$ can be transformed into a $S$-reduced word by a sequence of nil-moves and braid-moves without braid-moves that interfere in more than one index.\par
Braid-moves are executed on consecutive subwords of $\bs = (\bs(1),\dots,\bs(m))$. So $a'_y$ interferes with $a'_x$ in at least two adjacent indices that are both not $e$. This implies that $a'_x$ as well as $a'_y$ are braid-moves 
permuting the same two generators $s_i,s_j\in S$, because we have $m_{ij}>2$ for all distinct $s_i,s_j\in S$. We can assume that $a'_x$ and $a'_y$ are interfering in less than $m_{ij}$ indices because in this case $a'_y$ would just reverse $a'_x$ and both braid-moves can be omitted to obtain a reduced word. Thus, we additionally may assume that there exists a largest non-empty index $i_e$ touched by $a'_x$ and $a'_y$ that is adjacent to a larger index $i_f$ solely touched by $a'_y$ (all other cases are symmetric). Since $a'_y$ is executed after $a'_x$ in $(a'_k)_{k\in K}$, it follows $\bs(i_e) = \bs(i_f)\in \{s_i,s_j\}$. This allows us to define a new sequence $(a_l)_{l\in L}$ of nil-moves and braid-moves. The fist move $a_1$ is a nil-move that substitutes both adjacent entries $\bs(i_e)$ and  $\bs(i_f)$ with $e$. The resulting word $\bs'$ has $c-1$ non-empty indices. The rest of the sequence $(a_l)_{l\in L}$ is a sequence that reduces $\bs'$ such that all pairs of braid-moves interfere maximally in one index, which exists by the induction assumption. This completes the proof. 
\end{proof}

By using the lemmas above, we can derive a lower bound for the reflection length function from Theorem~\ref{Thm: Lower bound} for some elements. This is without knowing a deletion set, which is stronger than knowing the reflection length.\par
Therefore, consider a specific $S$-reduced word $\bs$ for an element $w$ in a Coxeter group $W$ and determine the maximal number of non-interfering braid-moves possible after omitting some letters in $\bs$. The lower bounds obtained this way aren't sharp generally. For Coxeter groups with sufficiently large braid relations, it implies that the powers of a Coxeter element have unbounded reflection length. This leads to the proof of Theorem~\ref{Thm: unbounded reflection length for powers of Coxeter elements}.
 
\begin{proof}[Proof of Theorem~\ref{Thm: unbounded reflection length for powers of Coxeter elements}]
Assume $k>2$.
For a natural number $\lambda\in \NN$, Theorem~\ref{Thm: solution word problem} implies that $\bs=(s_1\cdots s_n)^\lambda$ is the unique $S$-reduced word representing the element $w^\lambda$. Let $\bs'$ be the word obtained from $(s_1\cdots s_n)^\lambda$ by removing all letters in a deletion set $D(\bs)$. There exists a finite sequence $(\alpha_i)_{i\in I}$ of nil-moves and braid-moves on $\bs'$ to transform $\bs'$ into the empty word (see Theorem~\ref{Thm: solution word problem}). According to Lemma \ref{Lem: Always exists sequence with braid moves not interferring more than in one index}, we may assume additionally that all pairs of braid-moves in $(\alpha_i)_{i\in I}$ interfere maximally in one index of $\bs = (\bs(1), \dots, \bs(\lambda\cdot n))$.\par
We don't know a concrete deletion set for $\bs$ nor a sequence of nil-moves and braid-moves to reduce $\bs'$. Hence, it is necessary to consider the lowest lower bound obtainable with Theorem~\ref{Thm: Lower bound} for the reflection length of $w^\lambda$. A sequence of interfering braid-moves is counted as maximally four braid-moves for the lower bound of Theorem \ref{Thm: Lower bound} (see Lemma \ref{Lem: braid moves interfering in one letter} and Remark~\ref{Rem: braid moves interfering in one letter}). Therefore, we consider the maximal number of possible non-interfering braid-moves on $\bs$. Let $\xi(\lambda)$ be the maximal number of non-interfering braid-moves possible on $(s_1\cdots s_n)^\lambda$. Hence, a lower bound for the reflection length of $w$ is 
\begin{equation}\label{Eq: lower bound}
l_{R_n}(\omega_n((s_1\cdots s_n)^\lambda))-2 \xi(\lambda).
\end{equation} 
By counting letters, the minimal $S$-length  of a minimal consecutive subword in $(s_1\cdots s_n)^\lambda$ containing $\mathbf{b}_{ij}$ in dependency on $m_{ij}$ is
\[
\chi(m_{ij}) := 
	\begin{cases}
	\frac{m_{ij}-1}{2}\cdot n+1  & \text{for odd}\; m_{ij}\\
\frac{m_{ij}}{2}\cdot n-(n-2) & \text{for even}\; m_{ij}
	\end{cases}
\]
for all distinct $s_i, s_j\in S$. The minimal word length of a minimal consecutive subword in $\bs$ containing $\mathbf{b}_{ij}$ as a subword is obtained for $j=i+1$ (there are other minimal cases obtained by conjugacy). \par
The word length of $w^\lambda$ is $\lambda\cdot n$. Assume $k = m_{ij}$ with $j= i+1$. Thus, we have
\[
\xi(\lambda) = \frac{\lambda\cdot n}{\chi(m_{ij})}.
\]
Inserting this in Equation~\ref{Eq: lower bound} together with Lemma~\ref{Lem: l_R in univ. group, powers of Cox element} yields the following lower bound for the reflection length of $w^\lambda$:
\begin{align*}
 & (\lambda-1)\cdot(n-2)+n -2\cdot \frac{\lambda\cdot n}{\chi(m_{ij})}\leq l_R(w^\lambda) \\
\Leftrightarrow\quad  & \lambda n \left(1- \frac{2}{\chi(m_{ij})}\right)-2\lambda +2 \leq l_R(w^\lambda).
\end{align*} 
The left term is an unbounded monotonous growing function in the variable $\lambda$ if and only if
\[
\left(1-\frac{2}{\chi(m_{ij})}\right)> \frac{2}{n}.
\]
By inserting $\chi(m_{ij})$, we obtain that this is equivalent to the condition 
\begin{equation}\label{equation: for even mij}
\left(\frac{m_{ij}}{2}\cdot n-(n-2)\right)\cdot(n-2)-2n>0
\end{equation}
for even $m_{ij}$.
For odd $m_{ij}$, it is equivalent to the condition
\begin{equation}\label{equation: for odd mij}
\left(\frac{m_{ij}-1}{2}\cdot n+1\right)\cdot(n-2)-2n>0.
\end{equation}
This implies that the sequence $(l_R(w^\lambda))_{\lambda\in \NN}$ diverges towards $\infty$  if the inequality~\ref{equation: for even mij} is fulfilled for even $k$ or if the inequality~\ref{equation: for odd mij} is fulfilled for odd $k$. Accordingly, the sequence of powers of a Coxeter element in a Coxeter group of rank $3$ has unbounded reflection length if $k\geq 5$. This is to be seen by inserting the corresponding values in the inequality. For higher rank Coxeter groups, the reflection length of the powers of a Coxeter element is unbounded for $k\geq 3$. This proves the lemma.
\end{proof}

\begin{Rem}
The theorem above is independent from Duszenko's Theorem~\ref{Thm: Duszenko LR unbpunded}. So it can be seen as a constructive proof of the unboundedness of the reflection length function on Coxeter groups whose braid relations are large enough.\par
The theorem above does not cover all \ina Coxeter groups. It is reasonable to conjecture that the same statement also applies to the exceptions since we know that the reflection length is also an unbounded function on all \ina Coxeter groups (see Theorem~\ref{Thm: Duszenko LR unbpunded}). With Theorem~\ref{Thm: Dyer's Theorem}, we implemented an algorithm to compute the reflection length. This  conjecture is supported by all our calculations of the reflection length for $n=3$ and small $\lambda$ in the exceptional cases.
\end{Rem}

\begin{example}
The Coxeter group $W$ defined by the Coxeter graph 
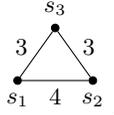
\[
\begin{tikzpicture}
\small
 \tikzset{enclosed/.style={draw, circle, inner sep=0pt, minimum size=.1cm, fill=black}}

      \node[enclosed, label={below,yshift=-0.0cm: $s_1$}] (A) at (0,0) {};
      \node[enclosed, label={below: $s_2$}] (B) at (1,0) {};
      \node[enclosed, label={above, yshift=0cm: $s_3$}] (C) at (0.5,0.7) {};

      \draw (A) -- (B) node[midway, below] (edge1) {$4$};
      \draw (C) -- (B) node[midway, right, ,yshift=+0.1cm] (edge2) {$3$};
      \draw (C) -- (A) node[midway, left, ,yshift=+0.1cm] (edge3) {$3$};
\end{tikzpicture}.
\]
is the Coxeter group of rank $3$ with the smallest braid relations and no commuting generators up to isomorphism. Table~\ref{Table: L_R in (3,3,4)} shows the reflection length of $\omega(s_1s_2s_3)^\lambda$ as a function of $\lambda$. For $\lambda \leq 2$, the reflection length of $\omega(s_1s_2s_3)^\lambda$ behaves in $W$ like $l_{R_n}(\omega_n(s_1s_2s_3)^\lambda)$ according to Lemma~\ref{Lem: r_R = univ. CoxGrp.}. For higher values of $\lambda$, it behaves differently. This is to be seen by comparing the values in Table~\ref{Table: L_R in (3,3,4)} with the values of the formula in Lemma~\ref{Lem: l_R in univ. group, powers of Cox element} for $\lambda\geq 3$. 

\end{example}

If $m_{ij}=2$ appears as a minimal braid relation however, the statement of Theorem~\ref{Thm: unbounded reflection length for powers of Coxeter elements} is false for rank $3$ Coxeter groups and and presumably more complicated for higher ranks. This is illustrated by the next lemma and its proof.

{
\renewcommand{\arraystretch}{1.5}
\setlength{\tabcolsep}{7pt}
\begin{table}
\centering
\begin{tabular}{c|c|c|c|c|c|c|c|c|c|c|c|c} 

 $\lambda$ & 2 & 3 & 4 & 5 & 6 & 7 & 8 & 9 & 10 & 11 & 12 & 15 \\ [0.4ex] 
 \hline
 $l_R(\omega(s_1s_2s_3)^\lambda)$ & 4 & 3 & 4 & 5 & 4 & 5 & 4 & 5 & 6 & 5 & 6 & 7 \\ [0.4ex] 
\end{tabular}
\vspace*{5pt}

\caption{Reflection length of powers of a Coxeter element.}
\label{Table: L_R in (3,3,4)}
\end{table}}

\begin{Lem}
Let $(W, S = \{s_1,s_2,s_3\})$ be a Coxeter system with two distinct commuting generators. For $\lambda\in \NN$, the powers of the Coxeter element $(s_1s_2s_3)^\lambda\in W$ have the following reflection length:
\[
l_R((s_1s_2s_3)^\lambda) =
	\begin{cases}
	2 & \text{for even}\;\lambda\\
	3\;\text{or}\; 1 & \text{for odd}\;\lambda.
	\end{cases}
\] 
\end{Lem}

\begin{proof}
Since two generators commute, there are distinct $1\leq i,j\leq 3$ with $m_{ij} =2$. Reflection length is invariant under conjugation (see Lemma~\ref{Lem: basic properties}). Exchanging the two generators that commute in $(s_1s_2s_3)^\lambda$ does not change the element, since they are adjacent and commute. This is why, we may assume that $m_{13} =2$ without loss of generality. \par
Take the following subdivision 
\[
(s_1s_2s_3)^\lambda = (s_1s_2s_3)^m \cdot (s_1s_2s_3)^{\lambda-m},
\]
where $m= \frac{\lambda}{2} -1$ if $\lambda$ is even and $m=\frac{\lambda-1}{2}$ if $\lambda$ is odd.
 Next, we apply the braid-move $s_3s_1\mapsto s_1s_3$ to all consecutive subwords $s_3s_1$ in the word $(s_1s_2s_3)^{\lambda-m}$. Braid-moves don't change the element represented by the words.
 When $\lambda$ is even, we obtain the reflection factorisation 
\begin{equation}
 (s_1s_2s_3)^m \cdot s_1s_2s_1\cdot (s_3s_2s_1)^{\lambda-m-2}\cdot s_3s_2s_3.
\end{equation}
In this case, it is $\lambda-m-2 = m$. When $\lambda$ is odd, we obtain the factorisation 
\begin{equation}
(s_1s_2s_3)^{m-1}\cdot s_1s_2s_3s_1s_2s_1\cdot (s_3s_2s_1)^{\lambda-m-2}\cdot s_3s_2s_3.
\end{equation}
In the latter case, it is $\lambda-m-2 = m - 1$.\par
From both factorisations, we can read off an upper bound for the reflection length of $(s_1s_2s_3)^\lambda$. The first factorisation is a reflection factorisation and the reflection length has the same parity as the word length (see Lemma~\ref{Lem: basic properties}). The word does not represent the identity, because it is reduced (see Theorem~\ref{Thm: Speyer's Theorem}). Hence, for even $\lambda$ it is $l_R( (s_1s_2s_3)^\lambda)= 2$.\newline
The second factorisation is not a reflection factorisation. Here, Theorem~\ref{Thm: Dyer's Theorem} gives us $l_R((s_1s_2s_3)^\lambda)\leq 3$ by removing the two letters in the middle of the word $s_1s_2s_3s_1s_2s_1$.
\end{proof}

\section{Upper bounds for reflection length in \sib Coxeter groups}
In this section, we study upper bounds of the reflection length of powers of Coxeter elements in \sib Coxeter groups. Therefore, we distinguish between rank $3$ and higher-rank Coxeter groups, because the reflection length of powers of elements with word length $3$ behaves differently from the reflection length in higher-rank Coxeter groups.\par

Lemma \ref{Lem: r_R = univ. CoxGrp.} is a first hint that not all braid-moves possible on a word are influencing the reflection length. Since braid-moves don't change the word length, every letter must be deleted by a nil-move. For every consecutive subword of a power of a Coxeter element $c^\lambda$ on which a braid-move was applied, letters to cancel this subword are to be found already in the right order in a reduced expression for $c^\lambda$. This is why, counting half of all possible braid-moves on a reduced expression of a power of a Coxeter element, is intuitive for estimating an upper bound for its reflection length. This section shows that this is indeed leading to a sharp upper bound for the reflection length.\par

Let $(W^n_k,S)$ be a \sib Coxeter system of rank $n\geq3$ with a constant edge labelling $m=k\in \NN_{\geq 4}$.   
Theorem~\ref{Thm: unbounded reflection length for powers of Coxeter elements} covers all \sib \ina Coxeter groups with one exception. This follows directly from the classification of Euclidean reflection groups.
\begin{Cor}
The powers of every Coxeter element in a \sib \ina Coxeter system $(W, S)$ have unbounded reflection length if $(W,S)$ is not the \sib Coxeter system  of rank $3$ with $m_{ij} = 4$. 
\end{Cor}

Moreover, we can directly extract a lower bound for the reflection length of the powers of a Coxeter element from the proof of Theorem~\ref{Thm: unbounded reflection length for powers of Coxeter elements}.  Sharp upper bounds are given by the following results.

\begin{Thm}\label{Thm: L_R for n=3}
The reflection length of elements of the form $(s_1s_2s_3)^\lambda s_1\cdots s_r$ in a \sib Coxeter system $(W^n_k,S)$ with $1\leq r  \leq 3$ and $k\geq 3,\;\;\lambda\in \mathbb{N}_0$ is bounded from above by
\[
l_R(\omega((s_1s_2s_3)^\lambda s_1\cdots s_r))\leq \lambda+r-2\cdot \floor*{\frac{\lambda+\id_{r\geq 2}}{k}}.
\]
\end{Thm}
\begin{proof}
We show the inequality by induction over $\lambda$.
We assume $k$ to be at least $3$. Thus, for $\lambda\leq 1$ we have  $\floor*{\frac{\lambda+\id_{r\geq 2}}{k}}=0$ and $(s_1s_2s_3)^\lambda s_1\cdots s_r$ has no subword of the form $(s_is_j)^3$. By Lemma \ref{Lem: r_R = univ. CoxGrp.} the reflection length $l_R(\omega((s_1s_2s_3)^\lambda s_1\cdots s_r))$ is equal to $l_{R_n}(\omega_n((s_1s_2s_3)^\lambda s_1\cdots s_r))$ for $\lambda\leq 1$ and Lemma \ref{Lem: l_R in univ. group, powers of Cox element} implies
\[
l_R(\omega((s_1s_2s_3)^\lambda s_1\cdots s_r))= \lambda+r.
\]
Hence, the statement of the lemma is true for $\lambda\leq 1$.\par
Now, consider $(s_1s_2s_3)^{\lambda+1}s_1$. Reflection length is invariant under conjugation and invariant under permutation of generators, since we only consider \sib Coxeter groups (see Lemma~\ref{Lem: basic properties}). Together with the induction hypothesis, this implies
\[
l_R(\omega((s_1s_2s_3)^{\lambda+1}s_1))= l_R(\omega((s_2s_3s_1)^\lambda s_2s_3))= \lambda+2-2\cdot \floor*{\frac{\lambda+1}{k}}.
\]
For the word $(s_1s_2s_3)^{\lambda+1}s_1s_2$, we obtain by conjugation, Lemma~\ref{Lem: basic properties} and by the induction hypothesis
\begin{align*}
l_R(\omega((s_1s_2s_3)^{\lambda+1}s_1s_2)) &= l_R(\omega(s_1s_2(s_1s_2s_3)^{k-2}s_1s_2 (s_3s_1s_2)^{\lambda+1-(k-1)}s_3))\\
 &\leq l_R(\omega(s_1s_2(s_1s_2s_3)^{k-2}s_1s_2))+l_R(\omega((s_3s_1s_2)^{\lambda+1-(k-1)}s_3))\\
 &= (k-2)+\lambda+1-(k-1)+1-2\floor*{\frac{\lambda+1-(k-1)}{k}}\\
 &= (\lambda+1)-2\floor*{\frac{(\lambda+1)+1-k}{k}}\\
 &= (\lambda+1)+2 -2\floor*{\frac{(\lambda+1)+1}{k}}.
\end{align*}
So the lemma holds in this case.\par 
For the reflection length of elements of the form  $(s_1s_2s_3)^{\lambda+1}s_1s_2s_3$ we have
\begin{align*}
l_R(\omega((s_1s_2s_3)^{\lambda+1}s_1s_2s_3)) &\leq l_R(\omega((s_1s_2s_3)^{\lambda+1}s_1s_2)) +1\\
 &= (\lambda+1)+2-2\floor*{\frac{(\lambda+1)+1}{k}}+1
\end{align*}
according to the induction hypothesis and Lemma~\ref{Lem: basic properties}. This completes the proof of the lemma. 
\end{proof}

We state the following theorem analogously to Theorem \ref{Thm: L_R for n=3} \sib Coxeter groups of rank for $n\geq 4$.

\begin{Thm}\label{Thm: L_R for n>=4}
The reflection length of the element represented by the word $\bs= (s_1s_2\cdots s_n)^\lambda s_1\cdots s_r$ in a \sib Coxeter system $(W,S)$ with $1\leq r  \leq n$ and $\lambda\in \mathbb{N}_0$ as well as $n\geq 4$ is bounded from above by
\[
l_R(\omega(\bs))\leq \lambda(n-2)+r-2\cdot\id_{(\lambda+\id_{r\geq 2})\geq k}\cdot \left(1+ \floor*{\frac{\lambda-k+\id_{r\geq 2}}{k-1}}\right).
\]
\end{Thm}
\begin{proof}
The assertion is proven by induction over $\lambda$.
If $(\lambda+\id_{r\geq 2})$ is strictly smaller than $k$, the theorem is true because of Corollary \ref{Cor: l_R= l_Rn in our Situation} and Lemma \ref{Lem: l_R in univ. group, powers of Cox element}. For the parameters we have $k\geq 3$ and $n\geq 4$.\par
Assume that the statement of the theorem is true for all $\lambda'< \lambda$ and all $1\leq r  \leq n$ with $\lambda, \lambda'\in \mathbb{N}$. For the reflection length of the element $w=\omega((s_1s_2\cdots s_n)^\lambda s_1)$ we have
\[
l_R(w)= l_R(\omega((s_2\cdots s_ns_1)^{\lambda-1}s_2\cdots s_n)),
\]
because conjugacy in general and permuting the generators in a \sib Coxeter group preserves reflection length (see Lemma~\ref{Lem: basic properties}).  The induction hypothesis gives us
\begin{align*}
&l_R(w)
= (\lambda-1)\cdot (n-2)+(n-1)-2\cdot \id_{(\lambda-1+1)\geq k}\left(1+ \floor*{\frac{\lambda-1-k+1}{k-1}}\right)\\
&= \lambda\cdot (n-2)+1-2\cdot \id_{(\lambda+\id_{r\geq 2})\geq k}\left(1+ \floor*{\frac{\lambda-k}{k-1}}\right).
\end{align*}
Thus, the theorem is true for $\lambda$ and $r=1$.\par 
To make a second induction argument according to $r$ for a fixed $\lambda$, assume that the statement of the theorem is true for all $\lambda'< \lambda$ and all $r'$ as well as for $\lambda'=\lambda$ and all $r'$ with $1<r'< r$. For the reflection length of the element $w=\omega((s_1s_2\cdots s_n)^\lambda s_1\cdots s_r)$, we have
\begin{align*}
l_R(w)= l_R(\omega(s_1\cdots s_r (s_1\cdots s_n)^{k-2} s_1s_2(s_3\cdots s_2)^{\lambda-(k-1)} s_3\cdots s_n)), 
\end{align*}
because conjugacy preserves reflection length. All exponents are non-negative since $(k-1)\leq \lambda$ (the other case is covered by the induction hypothesis). We obtain the identity element from the consecutive subword $s_1\cdots s_r (s_1\cdots s_n)^{k-2} s_1s_2$ if we remove all letters distinct from $s_1$ and $s_2$ in it. This is true, because $\omega((s_1s_2)^k)= \id$ in $W^n_k$. Hence, we have the following inequality
\begin{align*}
l_R(w)\leq  (k-2)\cdot (n-2) +(r-2)+l_R(\omega(s_3\cdots s_2)^{\lambda-(k-1)} s_3\cdots s_n)).
\end{align*} Permuting generators in a \sib Coxeter group does not change the reflection length (see Lemma~\ref{Lem: basic properties}). Together with the induction hypothesis, this implies
\begin{align*}
l_R(w)\leq\; & (k-2)\cdot (n-2) +(r-2) + (\lambda-(k-1))\cdot (n-2)+(n-2)\\
 & -2\cdot \id_{(\lambda-(k-1)+1)\geq k}\left(1+ \floor*{\frac{\lambda-(k-1)-k+\id_{r\geq 2}}{k-1}}\right)\\
 = \; &\; \lambda\cdot (n-2) + r -2 \\
  & -2\cdot \id_{(\lambda-(k-1)+1)\geq k}\left(1+ \floor*{\frac{\lambda-(k-1)-k+\id_{r\geq 2}}{k-1}}\right).
\end{align*}
There are two cases to be distinguished. The induction start is done for all words $(s_1s_2\cdots s_n)^{\tilde{\lambda}} s_1\cdots s_{\tilde{r}}$ with $(\tilde{\lambda}+\id_{\tilde{r}\geq 2})<k$. So we assume $(\lambda+\id_{r\geq 2})\geq k$.\par 
 In the case where $(\lambda-(k-1)+1)< k$, it follows that $\lambda-(k-1)<(k-1)$. This implies $\floor*{\frac{\lambda-k+1}{k-1}}= 0$ and we have 
\begin{align*}
l_R(w) &\leq \;  \lambda\cdot (n-2) + r -2\\
 &= \lambda(n-2)+r-2\cdot\id_{(\lambda+\id_{r\geq 2})\geq k}\left( 1+ \floor*{\frac{\lambda-k+\id_{r\geq 2}}{k-1}}\right).
\end{align*}
Else  $(\lambda-(k-1)+1)\geq k$ and it follows directly
\begin{align*}
l_R(w)&\leq\;  \lambda\cdot (n-2) + r -2 -2\cdot \left(1+ \floor*{\frac{\lambda-(k-1)-k+\id_{r\geq 2}}{k-1}}\right)\\
 &=  \lambda\cdot (n-2) + r -2\cdot \left(1+ \floor*{\frac{\lambda-k+\id_{r\geq 2}}{k-1}}\right)\\
 &= \lambda(n-2)+r-2\cdot\id_{(\lambda+\id_{r\geq 2})\geq k}\left(1 + \floor*{\frac{\lambda-k+\id_{r\geq 2}}{k-1}}\right).
\end{align*}
In total, the inequality
\[
l_R((s_1s_2\cdots s_n)^\lambda s_1\cdots s_r)\leq \lambda(n-2)+r-2\cdot\id_{(\lambda+\id_{r\geq 2})\geq k}\left(1+ \floor*{\frac{\lambda-k+\id_{r\geq 2}}{k-1}}\right)
\]
is proven by induction and thus the theorem, too.
\end{proof}

\begin{Rem}
The question of how to connect the upper and the lower bound remains. The lower bound obtained from the proof of Theorem~\ref{Thm: unbounded reflection length for powers of Coxeter elements} for elements of the form $w=\omega((s_1s_2\cdots s_n)^\lambda)$ in a \sib Coxeter group $W^n_k$ is
\[
(\lambda-1)\cdot(n-2)+n -2\cdot \frac{\lambda\cdot n}{\frac{k-1}{2}\cdot n+1}\leq l_R(w)
\]
for odd $k$. The negative term of his lower bound is roughly double the negative term in the upper bound from Theorem~\ref{Thm: L_R for n>=4} 
\[
l_R(w)\leq (\lambda-1)\cdot(n-2)+n-2\cdot\id_{\lambda\geq k}\cdot \left(1+ \floor*{\frac{\lambda-k}{k-1}}\right)
\]
(also true for Theorem~\ref{Thm: L_R for n=3}). For the lower bound, the negative part of the term counts subwords of the form $\mathbf{b}_{ij}$ of word length $k$. Whereas for the upper bound, the negative part of the term counts subwords of the from $(s_is_j)^k$ of word length $2k$.\par
Our computations of $l_R(\omega((s_1s_2\cdots s_n)^\lambda))$ for small $\lambda$ in different \sib Coxeter groups show in all instances that the upper bounds established in this section are exactly the reflection length.\par
Based on this and Lemma~\ref{Lem: r_R = univ. CoxGrp.}, we conjecture that the upper bounds from Theorem~\ref{Thm: L_R for n=3} and Theorem~\ref{Thm: L_R for n>=4} are equal to the reflection length function itself.
\end{Rem}

\section{The general relation between reflection length in arbitrary and universal Coxeter groups}
The results obtained in this work are mostly based on the comparison between the reflection length of elements that are represented by the same word in an arbitrary and in the universal Coxeter group of the same rank. 
The statement of the following conjecture implies a complete understanding of the relationship between the reflection length function in these different Coxeter groups. 

\begin{Con} \label{Conj: Dyer Conj}
Let $W=\langle S\mid\mathcal{R}\rangle$ be a Coxeter group and $w\in W$ be an element. Further, let $u_1\cdots u_p$ be a $S$-reduced expression for $w$ in $W$ with $u_i\in S$. There exists a letter $s$ in $u_1\cdots u_p$ such that omitting it results in 
\begin{align*}
l_R(\omega(u_1\cdots \hat s\cdots u_p))&=l_R(w)-1\text{ and }\\
l_{R_n}(\omega_n(u_1\cdots \hat s\cdots u_p))&=l_{R_n}(\omega_n(u_1\cdots u_p))-1.
\end{align*}
\end{Con}
A weaker version would be that every element has a $S$-reduced expression for which the statement of the conjecture is true. If the reflection length in $W$ is $1$, the conjecture is true. For the proof of this, we need the following definition.

\begin{Def}\label{Def: quasi-palindrome}
We define pairs of reduced words $(\bs_i, \bs_{-i})$ with $i \in \NN$ as words over an alphabet $S$ with relation $s=s^{-1}$ for all $s\in S$ such that one of the following conditions hold 
\begin{itemize}
\item[(i)]$\bs_{-i}=\bs_i^{-1}$,
\item[(ii)] for two letters $s_1, s_2\in S$ we have $\bs_i\in \{s_1, s_2\}^*$, $l_S(\bs_i)\geq 2$ and 
\[
\bs_{-i}= \begin{cases}
	\tau_{1,2}(\bs_i)\quad &\text{for odd word length}\\
	 \bs_i \quad &\text{for even word length}
	 \end{cases},
	 \]
where $\tau_{1,2}: S^*\to S^*$ exchanges $s_1$ and $s_2$.
\end{itemize}
For an $s\in S$ we define a \textit{twisted-palindrome} 
of odd word length to be a word 
\[
\bs_1\cdots \bs_n\cdot s \cdot \bs_{-n} \cdots \bs_{-1},
\]
where $(\bs_i, \bs_{-i})$ satisfies (i) or (ii) for all $1\leq i \leq n$. 
\end{Def}
\begin{Rem}
Twisted-palindromes are special cases of twisted conjugates of the generators in $S$ if $(W,S)$ is a universal Coxeter system. Condition $(i)$ and $(ii)$ are disjoint. 
\end{Rem}

\begin{Lem}\label{Lem: Removing middle from quasi-palindrome}
Let $\bs_1\cdots \bs_n\cdot s \cdot \bs_{-n} \cdots \bs_{-1}$ be a twisted-palindrome. For the element $t= \omega_n( \bs_1\cdots \bs_n\cdot s \cdot \bs_{-n} \cdots \bs_{-1})$ represented by this word the following equation holds:
\[
l_{R_n}(\omega_n(\bs_1\cdots \bs_n\cdot \hat{s} \cdot \bs_{-n} \cdots \bs_{-1}))= l_{R_n}(t)-1,
\]
where the hat over $s$ means omitting $s$. 
\end{Lem}
\begin{proof}
Since the element $t$ has odd word length in $W_n$, we have $l_{R_n}(t)=2k+1$ for $k\in \NN_0$.
We prove the statement by induction over $k\in \NN_0$. For $k=0$, $l_{R_n}(t)$ is equal to $1$ and we know that $\bs_1\cdots \bs_n\cdot s \cdot \bs_{-n} \cdots \bs_{-1}$ is a palindrome since there are no braid relations in $W_n$. Thus, we have $\bs_{-i}=\bs_i^{-1}$ for all $i\in \{1,\dots, n\}$ and therefore $l_{R_n}(\omega_n(\bs_1\cdots \bs_n\cdot \hat{s} \cdot \bs_{-n} \cdots \bs_{-1}))=0$. \par
In general, if we have $l_{R_n}(\bs_i\bs\bs_{-i})\geq l_{R_n}(\bs)+2$ for a quasi-palindrome $\bs$ and a pair of words $(\bs_i, \bs_{-i})$ like in Definition \ref{Def: quasi-palindrome}, it follows that the inequality is an equality. The reflection length has to increase by an even number, because of parity reasons. The reflection length increases maximally by two since a word of the form $\bs_i\bs_{-i}$ has maximal reflection length $2$.\par 
The pair satisfies condition (ii) from Definition \ref{Def: quasi-palindrome}, since the two conditions are disjoint and conjugacy preserves reflection length (Lemma~\ref{Lem: basic properties}). To neutralize this effect on the reflection length of adding a pair of words $(\bs_i, \bs_{-i})$ in the outlined way, we distinguish two cases. If the word length of $\bs_i$ is odd and hence also the word length of $\bs_{-i}$, we remove the middle of both words and they vanish completely by applying $s=s^{-1}$ for all $s\in S$. Otherwise, if the word length of both is an even number, we remove the first letter in both cases and obtain a word that is conjugated to $\bs$. Conjugation preserves reflection length. We especially obtain a quasi-palindrome again. \par
Let $l_{R_n}(\bs_1\cdots \bs_n\cdot s \cdot \bs_{-n} \cdots \bs_{-1})=2k+1+2$. From above, we know that there exists a pair of words $(\bs_i, \bs_{-i})$ such that the reflection length will be reduced by two after omitting one letter in each word as described above. Moreover, we obtain an odd quasi palindrome with reflection length $2k+1$ and we can apply the induction assumption that by deleting the letter in the middle of the word we reduce the reflection length by one again. The deleted letters are elements in a deletion set of $\bs_1\cdots \bs_n\cdot s \cdot \bs_{-n} \cdots \bs_{-1}$ in $W_n$. According to Lemma \ref{Lem: After Dyer}, we have $l_{R_n}(\bs_1\cdots \bs_n\cdot \hat{s} \cdot \bs_{-n} \cdots \bs_{-1})= 2k+1$ and the induction is complete.
\end{proof}

\paragraph{\bf{Acknowledgements}} The idea for this project came during a two-month research stay at the University of Sydney with Anne Thomas founded by the
Deutsche Forschungsgemeinschaft (DFG, German Research Foundation) – 314838170, GRK 2297
MathCoRe. The author thanks Anne Thomas and the University of Sydney for their hospitality as well as the DFG. Further, the author thanks Petra Schwer for many helpful suggestions
during the preparation of the paper.

\bibliographystyle{amsplain}
\bibliography{Literatur.bib}

\providecommand{\bysame}{\leavevmode\hbox to3em{\hrulefill}\thinspace}
\providecommand{\MR}{\relax\ifhmode\unskip\space\fi MR }
\providecommand{\MRhref}[2]{%
  \href{http://www.ams.org/mathscinet-getitem?mr=#1}{#2}
}
\providecommand{\href}[2]{#2}
\begin{thebibliography}{10}

\bibitem{Bjorner2005}
Francesco~Brenti Anders~Bjorner, \emph{Combinatorics of {C}oxeter groups},
  Springer-Verlag, 2005.

\bibitem{Carter1972}
R.~W. Carter, \emph{{Conjugacy classes in the Weyl group}}, Compositio
  Mathematica \textbf{25} (1972), no.~1, 1--59.

\bibitem{Coxeter1934a}
H.~S.~M. Coxeter, \emph{Discrete groups generated by reflections}, Annals of
  Mathematics \textbf{35} (1934), no.~3, 588--621.

\bibitem{Davis2012}
Michael Davis, \emph{The geometry and topology of {C}oxeter groups}, Princeton
  University Press, Princeton, 26 Nov. 2012.

\bibitem{Dehn1911}
M.~Dehn, \emph{Über unendliche diskontinuierliche {G}ruppen}, Mathematische
  Annalen \textbf{71} (1911), no.~1, 116--144.

\bibitem{Drake2021}
Brian Drake and Evan Peters, \emph{{An upper bound for reflection length in
  Coxeter groups}}, Journal of Algebraic Combinatorics \textbf{54} (2021),
  no.~2, 599--606.

\bibitem{Duszenko2011}
Kamil Duszenko, \emph{{Reflection length in non-affine Coxeter groups}},
  Bulletin of the London Mathematical Society \textbf{44} (2011).

\bibitem{Dyer2001}
M.J. Dyer, \emph{{On minimal lengths of expressions of Coxeter group elements
  as products of reflections}}, Proceedings of the American Mathematical
  Society \textbf{129} (2001).

\bibitem{Humphreys1990}
James~E. Humphreys, \emph{Reflection groups and {C}oxeter groups}, Cambridge
  Studies in Advanced Mathematics, Cambridge University Press, 1990.

\bibitem{Lewis2018}
Joel~Brewster Lewis, Jon McCammond, T.~Kyle Petersen, and Petra Schwer,
  \emph{{Computing reflection length in an affine Coxeter group}}, Transactions
  of the American Mathematical Society \textbf{371} (2018), no.~6, 4097–4127.

\bibitem{Speyer2008}
David Speyer, \emph{Powers of {C}oxeter elements in infinite groups are
  reduced}, Proceedings of the American Mathematical Society \textbf{137}
  (2008), no.~4, 1295--1302.

\bibitem{Tits1969}
Jacques Tits, \emph{{Le} {probl}{\`{e}}{me} {des} {mots} {dans} {les} {groupes}
  {de} {Coxeter}* i risultati di questo lavoro sono stati esposti nella
  conferenza tenuta il 14 dicembre 1967.}, Symposia Mathematica, Teoria dei
  gruppi, INDAM, Roma, 13–16 dicembre 1967 \textbf{1} (1969), 175--185.

\end{thebibliography}
\end{document}